\documentclass[11pt]{amsart}

\usepackage{latexsym,amssymb,enumerate}
\newtheorem{propo}{Proposition}[section]

\newtheorem{lemma}[propo]{Lemma}
\newtheorem{cor}[propo]{Corollary}

\newtheorem{corollary}[propo]{Corollary}

\newtheorem{theorem}[propo]{Theorem}

\numberwithin{equation}{section}

\newcommand{\CC}{{\mathbb C}}
\newcommand{\RR}{{\mathbb R}}

\newcommand{\FF}{{\mathbb F}}

\newcommand{\GL}{\mathrm {GL}}

\newcommand{\GU}{\mathrm {GU}}
\newcommand{\SU}{\mathrm {SU}}

\begin{document}

\title{On the singular value decomposition over finite fields and orbits of $GU \times GU$}

\author{Robert M. Guralnick}
\address{Department of Mathematics, University of Southern California,
Los Angeles, CA 90089-2532, USA}
\email{guralnic@usc.edu}

\date{\today}

\dedicatory{Dedicated to the memory of Tonny Springer}

\thanks{The  author was partially supported by the NSF
grant DMS-1901595 and a Simons Foundation Fellowship.}

\begin{abstract}   The singular value decomposition of a complex matrix is a fundamental concept in linear algebra and has proved 
extremely useful in many subjects. 
It is less clear what the situation is over a finite field.   In this paper, 
we classify the orbits of $\GU_m(q) \times GU_n(q)$
on $M_{m \times n}(q^2)$ (which is the analog of the singular value decomposition).   The proof involves Kronecker's theory of pencils and the 
Lang-Steinberg theorem for algebraic groups.
 Besides the motivation mentioned above, this problem came up
in a recent paper of Guralnick, Larsen and Tiep \cite{GLT}  where a 
  concept of character level for the complex irreducible characters of finite, general or special,
linear and unitary groups was studied and bounds on the number of orbits was needed.
A consequence of this work determines possible pairs of Jordan forms for nilpotent matrices
of the form $AA^*$ and $A^*A$ over a finite field and $AA^{\top}$ and $A^{\top}A$ over arbitrary fields. 
\end{abstract}

\maketitle

\tableofcontents

\section{Introduction}

The singular value decomposition of a complex $m \times n$ matrix  $A$ is a very fundamental
topic and has applications in many areas of mathematics.  The eigenvalues of 
$A^*A$ are all non-negative real numbers.   The singular values are the positive square roots of the nonzero eigenvalues of $A^*A$ 
(which are the same as those of $AA^*$).  
   This can be interpreted in terms of orbits of
$\GU_m(\CC) \times GU_n(\CC)$ acting on $M_{m \times n}(\CC)$ (the action is given by
$A \rightarrow XAY^*$ ).   The result is that the orbits are in bijection with the similarity
classes of $A^*A$ (either under $\GU$ or $\GL$) 
and each orbit contains a unique  matrix where the entries on
the diagonal are non-negative and  non-increasing and all other entries are $0$.

This fundamental result is not hard to prove but it depends on the fact that any complex Hermitian matrix is diagonalizable with all eigenvalues real.    The results hold for any algebraically closed field of characteristic $0$ with respect to a (positive definite) inner product (and any two such are equivalent). 

For many reasons, the result becomes much more complicated if the field is not algebraically closed or the characteristic is not $0$.  Even over $\mathbb{C}$, the result is more complicated if
we consider Hermitian forms which are not definite (in particular, the stabilizer of
the form is not compact).  In this paper, we  are  interested in an analog for  finite 
fields.  In this case any two nondegenerate Hermitian forms are equivalent.    This was asked about on math.overflow in
November 2009
(and the answers indicated incorrectly that there was not much to say).

Fix a prime power $q$ and let $F$ be the field of $q^2$ and $F_0$ the subfield with $q$ elements.
Let $K$ be an algebraic closure of $F$.
We will completely characterize the orbits of $\GU_m(q) \times GU_n(q)$ on $M_{m \times n}(q^2)$
and in particular give formulas for the number of orbits.   This problem seems quite natural from a linear algebra point of view and was used in
a recent paper of Guralnick, Larsen and Tiep \cite{GLT} in proving some results about complex irreducible characters of $\GU_m(q)$ and
$\SU_m(q)$.

The problem essentially reduces to two cases.  The first is the case $m=n$ and $A$ is
invertible.  In this case, it is easy to see that an analog of the complex case holds:  
the orbits are determined by the conjugacy class of $A^*A$.   Thus, 
the orbits of invertible matrices are in bijection with the conjugacy classes of 
$\GL_n(q)$ as well as the unitary orbits on nondegenerate Hermitian matrices.  

The second case is when $A^*A$ is nilpotent (for definite Hermitian forms over $\mathbb{C}$,
this  does not occur unless $A=0$).  In this case, the number of orbits, 
$f(m,n)$, is independent of $q$ (and the characteristic).  We will give a formula for
$f(m,n)$ in terms of  partitions. In this case, it is not true that the orbit
is determined by the similarity classes of $A^*A$ (or even the pair of similarity
classes $A^*A$ and $AA^*$).   

The proof has two basic ingredients.  The first is the classification of the representations
a certain quiver (closely related to Kronecker's theory of matrix pencils).   This gives the
orbits of the action of $\GL_m(K)  \times GL_n(K)$ on $M_{m \times n}(K) \times M_{n \times m}(K)$ 
(where the action is given by $(A,B) \rightarrow (XAY^{-1}, YBX^{-1})$).  We then use the Lang-Steinberg theorem (and the fact that stabilizers are connected) to determine the orbits of 
$\GU_m(q) \times \GU_n(q)$ on $M_{m \times n}(q^2)$.   

We now state our main results.  It is slightly more convenient to state some of the results in terms of
linear transformations rather than matrices.   So fix a pair of vectors spaces $W, V$ over $F$
of dimensions $m$ and $n$ with nondegenerate Hermitian forms on each.   If $T:V \rightarrow W$
is an $F$-linear transformation, then we can define $T^*: W \rightarrow V$ 
by $(Tv, w)= (v, T^*w)$ for $v \in V$ and $w \in W$ (the first inner product is defined on $W$ and
the second on $V$).  If we choose orthonormal bases for $V$ and $W$, then the matrix of $T^*$
is obtained by the matrix of $T$ in the usual way.  

The first (quite elementary) result shows how we reduce to the two cases.

\begin{theorem} \label{reduce1}.  Let $T: V \rightarrow W$ be a linear transformation.
Then there are canonical decompositions  $V=V_0  \perp V_1$ and $W=W_0 \perp W_1$ such
that 
\begin{enumerate}
\item $\dim V_0= \dim W_0$,  $TV_0 = W_0$,  $T^*W_0 = V_0$
 (so $T^*T$ is bijective on $V_0$);
\item $TV_1 \subset W_1$, $T^*W_1 \subset V_1$ and $T^*T$ is nilpotent
on $V_1$.
\end{enumerate}
\end{theorem}

The invertible case is quite easy.   There is no harm in assuming that $W=V$ and
the Hermitian forms are the same.    If $G$ is a finite group, let $k(G)$ denote
the number of conjugacy classes of $G$. 

\begin{theorem}  \label{invertible} Let $T: V \rightarrow V$ be an invertible linear transformation.  
\begin{enumerate}
\item  $T^*T$ is conjugate to an element of $\GL_n(q)$;
\item  Any element of $\GL_n(q)$ is conjugate to $T^*T$ for some $T \in GL_n(q^2)$;
\item  The map $T \rightarrow T^*T$ induces
bijections between the orbits of $\GU_n(q) \times GU_n(q)$ on $\GL_n(q^2)$
and conjugacy classes of $\GL_n(q)$ and $GU_n(q)$ orbits on invertible Hermitian
matrices;
\item The number of orbits is $k(GL_n(q))$.    
\end{enumerate}
\end{theorem}

We will describe the nilpotent orbits explicitly, but for now we just note:

\begin{theorem} \label{nilpotent}    The number of orbits of $\GU(V) \times \GU(W)$
on $F$-linear maps $T$ from $V$ to $W$ such that $T^*T$ is nilpotent is a function $f(m,n)$ independent of $q$.
\end{theorem}

We give a formula $f(m,n)$ in Section \ref{desc:orbits}.   It is easy to see that  
we have $f(m,d)=f(m,2m)$ for any $d \ge 2m$ (and clearly $f(m,n)$ does not decrease
as $n$ increases).   Also, it follows that all such orbits have representatives
over the prime field.

Combining the two previous results gives:

\begin{cor} \label{count}    Assume that $m \le n$.
The number of orbits of $\GU_m(q) \times GU_n(q)$ is
$$
\sum_{i=0}^m     f(m-i, n-i) k(\GL_i(q)).
$$
\end{cor}

Note that   $k(GL_i(q))$ is  very close   to $q^i$ (but less than $q^i$)  and in particular is
a monic polynomial in $q$ of degree $i$ \cite{MR}.   Thus, the total number of orbits is a monic
polynomial in $q$ of degree $m$.  

Next, we note an analog of Flanders' theorem (see below for the statement).
If $N$ is a nilpotent square matrix of size $n$,  we let
$\mu(N) = \mu_1 \ge \ldots \ge \mu_r$ be the partition of $n$ associated to the
Jordan form of $N$  (for convenience of notation we allow the possibility of $\mu_j = 0$).

The theorem of Flanders  \cite{Fl}
(see also \cite{Jo}) is:

\begin{theorem} {\rm [Flanders]}  \label{AB versus BA} Let $L$ be a field and $m, n$ positive
integers.    Let $\lambda$ be a partition of $m$ and $\mu$ a partition of $n$.
There exist $A \in M_{m \times n}(L)$ and $B \in M_{n \times m}(L)$ with $AB$ nilpotent
such that 
$\mu(AB) = \lambda$ and $\mu(BA)= \mu$ if and only if $|\lambda_i - \mu_i| \le 1$
for all $i$.
\end{theorem}

Note that this result follows immediately from the classification of the representations
of the appropriate quiver (and in particular from Kronecker's theorem).  

We consider the same problem for   $A^*A$ and $AA^*$.  We note that not every possible
pair of partitions allowed by Flanders' Theorem can occur.  For example, if $d > 1$ and
$A$ is a $d \times d$ matrix over a finite field, then it cannot be the case that both
$AA^*$ and $A^*A$ are nilpotent matrices with a single Jordan block.
    In order to state the  result, we need to introduce some notation.

Let $\alpha$ be a partition of $m$ and $\beta$ a partition on $n$.  Let $\gamma
=\gamma(\alpha, \beta)$
be the partition where parts are the $\alpha_i = \beta_i > 0$.   If
$\lambda = \lambda_1 \ge \ldots \ge \lambda_r > 0$,  a component of $\lambda$
is a maximal consecutive subsequence of the parts of $\lambda$ such that
the difference of any two consecutive pieces is at most $1$.

\begin{theorem} \label{A*A versus AA*}   Let $\alpha$ be a partition of $m$
and $\beta$ a partition of $n$.       There exists
$A \in M_{m \times n}(q^2)$ so that $AA^*$ is nilpotent and
$\mu(A^*A)=\beta$ and $\mu(AA^*) = \alpha$
if and only if
\begin{enumerate}
\item $|\alpha_i - \beta_i| \le 1$ for all $i$; and
\item Every connected component of $\gamma(\alpha, \beta)$ not involving $1$ has even length.
\end{enumerate}
\end{theorem}

If $A$ is a matrix over a finite field $F_0$, let $F$ be a quadratic field extension.
Then viewing $A$ as a matrix over $F$, we have $A^{\top} = A^*$.   Thus, the
previous result applies to $A^{\top}A$ as well (over finite fields).  By standard
arguments, 
this gives the following result over an arbitrary field. 

\begin{corollary}\label{ATA versus AAT}   Let $K$ be a field. 
Let $\alpha$ be a partition of $m$
and $\beta$ a partition of $n$.       If there exists 
$A \in M_{m \times n}(K)$ so that  $AA^{\top}$ is nilpotent 
and $\mu(A^{\top}A)=\beta$ and $\mu(AA^{\top}) = \alpha$, then
\begin{enumerate}
\item $|\alpha_i - \beta_i| \le 1$ for all $i$; and
\item Every connected component of $\gamma(\alpha, \beta)$ not involving $1$ has even length.
\end{enumerate}
\end{corollary}

Over an arbitrary field, given two partitions satisfying the two conditions
above, one cannot in general find an $A$ with $A^{\top}A$ nilpotent 
and $\alpha(A^{\top}A)$ and $\beta(AA^{\top})$ as specified (e.g. over $\RR$
but also over any field where $-1$ is not a square).   It is not hard to show
that if $-1$ is a square in $F$, then the converse does hold and in particular
holds over an algebraically closed field.

The paper is organized as follows.  We first recall some basic results
about Hermitian matrices over a finite fields.   In the following section, we
recall some general results about quivers with particular attention to 
the quiver associated to Kronecker's theory.   We then recall some basic results
about algebraic groups and passing between finite groups and algebraic groups. 
In the next section, we prove the main results and give a formula for the number
of orbits.  In the final section, we make some further remarks.  

\section{Hermitian Matrices over Finite Fields}. \label{hermitian}

We recall some easy facts about Hermitian matrices over finite fields.
  See \cite[Lemma 3.1]{FG}.   We give a quick proof. 

\begin{lemma} \label{hermitian1}  Let $A \in M_n(q^2)$.   Then the following are equivalent.
\begin{enumerate}
\item $A$ is similar to a Hermitian matrix.
\item  $A$ is similar to a matrix in $M_n(q)$.
\end{enumerate}
Moreover, if $A$ and $B$ are Hermitian matrices, they are similar as 
matrices if and if only if they are conjugate by an element of $\GU_n(q)$.
\end{lemma}

\begin{proof}   If $A$ is similar to a Hermitian matrix, then it is similar to its
conjugate via the $q$-Frobenius map (because any matrix is similar to its
transpose).   Thus, all the invariant factors of $A$ are defined over $\FF_q$,
whence $A$ is similar to a matrix in $M_n(q)$.   

Suppose that $A$ is similar to $B \in M_n(q)$.     Then there exists a symmetric
matrix $S \in GL_n(q)$ with $SB=B^{\top}S=B^*S$ \cite{TZ}.   Thus, $B$ is self adjoint
with respect to the Hermitian form induced by $S$, whence $B$ is conjugate to
a Hermitian matrix.

The last statement follows by the Lang-Steinberg theorem (see below) and is proved
in \cite[Lemma 3.1]{FG}.  
\end{proof}

The next result follows from the   the primary decomposition for a matrix
and the previous result that any two Hermitian matrices which have the same rational
canonical form are conjugate via an element of the unitary group.

\begin{cor}  Let $V$ be a vector space over the field $F$ of $q^2$ elements
equipped with a nondegenerate Hermitian form.  Let $F_0$ be the subfield
of size $q$.   Let $T$ be a Hermitian linear
transformation on $V$ (i.e. self adjoint with respect to the Hermitian form).  
Then there is a unique decomposition $V=V_1 \perp V_2  \perp \ldots \perp V_r$ where
$T(V_i)=V_i$ and the characteristic polynomial of $T$ on $V_i$ is a power of
the monic irreducible polynomial $f_i(x) \in F_0[x]$ with $f_i(x) \ne f_j(x)$ for
$i \ne j$.  In particular, there is a unique decompostion $V= X \perp Y$
with $TX=X$ and $TY \subset Y$ with $T$ nilpotent on $Y$.  
\end{cor}

The next result follows easily from the fact that any nondegenerate 
form has an orthonormal basis.  

\begin{lemma} \label{hermitian2}  If $D$ is a rank $r$ Hermitian matrix in $M_{m}(q^2)$,
then $D= X^*X$ for some $X \in M_{r \times m}(q^2)$. 
\end{lemma}

\section{Quivers} \label{quivers}. 

We first recall some facts about quivers and their representations.  We refer
the reader to \cite{CB, DW, DR, Sc} for more details.  All quivers considered will be finite.

Fix a field $k$ and let $Q$ be a  quiver (i.e. a finite directed graph).  
Let $I = \{1, \ldots, n\}$ be the vertex
set of the graph and let $E$ denote the set of edges of $Q$.

A (finite dimensional) representation of $Q$ (over $k$) is an assignment of finite dimensional
vector spaces $V_i, i \in I$ and for each edge $e \in E$ from $i$ to $j$, a linear transformation
$T_e$ from $V_i$ to $V_j$.     The dimension vector of such a representation is
the vector $(\dim V_1, \ldots, \dim V_n)$.   If we fix the dimension vector, there is no
harm in considering only representations where the vector spaces $V_i$ are fixed.

Note that $G:=\GL(V_1) \times \ldots \times GL(V_n)$ acts the set of representations 
(for fixed $V_i$)
in the obvious way.  Two representations are equivalent if they are in the same $G$-orbit.
Note also that the representations of $Q$ are in bijection with modules over the path algebra
$k[Q]$ of the quiver.   If we consider representations with a fixed dimension vector, we
can then fix the spaces $V_i$.  In that case,   there is a bijection
between modules  and  $G$-orbits.   The next result is a special
case of the Noether-Deuring Theorem (see \cite[19.25]{La}).

\begin{lemma}   Let $Q$ be a quiver and let $k$ be a field.  If two representations
of $Q$ defined over $k$ become isomorphic over some field extension $k'$ of $k$, then
they are already isomorphic over $k$.
\end{lemma}

Here are some other basic facts that we will require.

\begin{lemma}   If $k$ is algebraically closed and $Q$ is a  quiver,
 then the centralizer in $G$ of a given
representation of a quiver $Q$ is a connected algebraic group.
\end{lemma}

\begin{proof}    We first work in $A:=\mathrm{End}(V_1) \times \ldots \times  \mathrm{End}(V_n)$.
Note that the centralizer of a representation in this algebra is a subalgebra of
$A$  (i.e. the endomorphism ring of the module over the path algebra)
and the centralizer is $G$ is the unit group of this algebra and so is connected
(since it is Zariski dense in the subalgebra).
\end{proof}

The next result shows that reversing the direction of edges does not really make
a big difference in the representation theory. 

\begin{lemma} \label{counting}
If $k$ is a finite field and $Q$ is a quiver, then the number of equivalence
classes of representations of $Q$ for a fixed dimension vector depends only on
$Q$ as an undirected graph.
\end{lemma}

\begin{proof} We may fix the $V_i, i \in I$.   Let $g:=(g_1, \ldots, g_n) \in G$.
The representations of $Q$ with the $V_i$ fixed is a linear space and as a $kG$-module
is isomorphic to a direct sum of $V_i \otimes V_j^*$  (one term for each edge from
$i$ to $j$).   

If we reverse one edge it has the effect of replacing one $V_i \otimes V_j^*$ by
$V_j \otimes V_i*$.

We claim that the number of fixed points of $g$ depends only $Q$ as undirected
graph -- i.e. the direction of each edge does not change this number.  It suffices
to consider the case where we reverse precisely one arrow.  Since the $kG$-modules
are a direct sum,  we see the problem reduces to showing that $(g_i, g_j)$ has
the same number of fixed points on $V_i \otimes V_j^*$ and $V_j \otimes V_i^*$.

This follows from the more general fact that we have a finite group $H$
and $W$ is a finite dimensional $kH$-modules, then for any $h \in H$, 
$h$ has the same number of fixed points on $W$ as on $W^*$.   This is because
the dual representation is obtained by applying transpose inverse to the original
representation and a matrix and its transpose are conjugate in $\GL$.

Now the result follows by Burnside's Lemma (due to Frobenius):  the number of orbits
of a finite group acting on a finite set is the average number of fixed points
of an element in the group. 
\end{proof}

We are interested in a very special easy case.   We first consider the quiver
$Q$ with two vertices $\{1, 2 \}$ with two edges each going from $1$ to $2$.
In this case, the representations of $Q$ correspond to pairs (or pencils) of linear transformations
$S, T$ from $V_1$ to $V_2$ and were classified by Kronecker (see \cite{CP, Ga}).  
Let $d_i=\dim V_i$.   

In particular the indecomposable representations are (up to equivalence) the following:

\begin{enumerate}
\item   $d_1=d_2=d$,  $S=I$ and $T$ is the companion matrix of a monic polynomial
$f(x)$ which is a  power of an 
irreducible polynomial in $k[x]$
(other than $x$); 
\item   $d_1=d_2=d$,  $S=I$ and $T$ is a single nilpotent Jordan block;
\item   $d_1=d_2=d$,  $T=I$ and $S$ is a single nilpotent Jordan block;
\item  $d= d_2=d_1 -1$ with $d_1 \ge 1$.  Let $e_0, \ldots, e_d$ be a basis
for $V_1$ and $f_1, \ldots, f_d$ be a basis for $V_2$.  Let   $Se_i=f_i$ for $i > 0$
and $Se_0=0$  and $Tf_i=e_{i-1}$;
\item   $d=d_1=d_2-1$ and $S$ and $T$ are interchanged from the previous case.
\end{enumerate}

For our purposes,  we want to study the quiver $Q$ with two vertices $1, 2$ 
over a field $k$ and 
two maps $S:V_1 \rightarrow V_2$ and $T:V_2 \rightarrow V_1$.

It is known that representations of the two quivers have almost
identical descriptions.  We deduce the second case from the first.  We will work
over a finite field (but one can extend the analysis to the case of general fields).

Now we write down some indecomposable modules for our second quiver. 
They are quite closely related to the first.   We will show that this is a complete
set of indecomposables.

\begin{theorem}  Let $k$ be a finite field.
The indecomposable finite dimensional modules for $k[Q]$ with
$Q$ as above are:
\begin{enumerate}
\item  $d_1=d_2=d$,  $S=I$ and $T$ is the companion matrix of a monic
polynomial $f(x) \in k[x]$ that is power of an 
irreducible polynomial in $k[x]$
(other than $x$); 
\item   $d_1=d_2=d$,  $S=I$ and $T$ is a single nilpotent Jordan block;
\item   $d_1=d_2=d$,  $T=I$ and $S$ is a single nilpotent Jordan block;
\item   $d= d_2=d_1 -1$ with $d_1 \ge 1$.  Let $e_0, \ldots, e_d$ be a basis
for $V_1$ and $f_1, \ldots, f_d$ be a basis for $V_2$.  Let   $Se_i=f_i$ for $i > 0$
and $Se_0=0$  and $Tf_i=e_{i-1}$;
\item   $d=d_1=d_2-1$ and $S$ and $T$ are interchanged from the previous case.
\end{enumerate}
\end{theorem}

\begin{proof}  It is easy to see that the above representations are indecomposable
(consider the products $ST$ and $TS$ -- they are indecomposable in each case)
and there is an obvious  bijection between the indecomposables given above and for the
other quiver with two arrows in the same direction.  Note this bijection preserves
dimension vectors.   Since any representation is (up to equivalence) uniquely a direct
sum of indecomposable representations, it follows that the number of inequivalent 
representations 
of $Q$ with a fixed dimension vector that involves only the indecomposables described
above is the same as the total number of inequivalent representations of the
Kronecker quiver with the same dimension vector.  Now by Lemma \ref{counting},
this accounts for all representations and the result holds.
\end{proof}

In fact, the assumption that $k$ is finite is superfluous.  This can be seen
by splitting the problem into two parts -- the case where $\dim V_1= \dim V_2$
and $ST$ is invertible and the case where $ST$ is nilpotent.  We will not pursue
this since our interest is in the case of finite fields.  

We denote these indecomposables by $I[f], J[d], J'[d], R[d]$ and $R'[d]$.
The two obvious invariants associated to a representation are the dimension
vector and the similarity classes of $ST$ and $TS$ (and these determine
the dimension vector).   For the indecomposables,
these uniquely determine the module aside from $J[d]$ and $J'[d]$.

\section{Algebraic Groups and Descent}  \label{descent}

We recall some basic facts about algebraic groups over an algebraically
closed field $K$  of characteristic $p > 0$.   We refer the reader to 
\cite{Sp, SS, St}.    If $\sigma$ is endomorphism of an algebraic group $G$, 
we let  $G_{\sigma}$ be the fixed point set of $\sigma$ on $G$.   The next result
was proved in a special case by Lang and in general by Steinberg \cite{St}.  

\begin{theorem} \label{lang-steinberg}  Let $G$ be a connected algebraic
group over $K$.  If $\sigma$ is an endomorphism of $G$ with  $G_{\sigma}$
finite, then the map  $g \rightarrow g^{-1} \sigma(g)$ is surjective.
\end{theorem}

It also follows that any endomorphism $\sigma$ with only a finite number of fixed
points is an abstract automoprhism of $G$.   The previous result can be
reformulated in the following way.    Set $H = \langle G, \sigma \rangle$.   Then all elements in the coset $\sigma G =G \sigma$
are conjugate via an element of $G$. 

If a group $H$ acts on a set $W$ and $w \in W$, we let $C_H(w)$ be the stabilizer 
of $w$ in $H$.   The next well known result is an easy consequence of the Lang-Steinberg theorem.

\begin{corollary} \label{cor:orbits}
Let $G$ be a connected algebraic group over $K$ and let $\sigma$
be an endomorphism of $G$ with $G_{\sigma}$ finite.  Let $H:= \langle G, \sigma \rangle$.
Assume that $H$ acts on an irreducible algebraic variety $Z$ with $G$ and $\sigma$
acting algebraically.   Let $Y \subseteq X$ be a $G$-orbit. 
\begin{enumerate}
\item   $\sigma$ fixes $Y$ if and only $\sigma(y)=y$ for some $y \in Y$.
\item   Let $Y$ be a $\sigma$ invariant orbit of $G$ on $X$.
 If $C_G(y)$ is connected for $y \in Y$,  then $Y_{\sigma}$ is a 
single $G_{\sigma}$ orbit.
\end{enumerate}
\end{corollary}

\begin{proof}   Clearly, if $\sigma$ fixes a point of $Y$, it fixes the orbit
(as it permutes the orbits).   Conversely if $\sigma y \in Y$, for some $y \in Y$,
then $\sigma y = gy$ for some $g \in G$, whence $g^{-1}\sigma$ has a fixed point
on $Y$.  Since $g^{-1}\sigma$ is conjugate to $\sigma$, $\sigma$ has
a fixed point on $Y$ as well.

Let $Y$ be a $\sigma$ invariant orbit and $y, z \in Y_{\sigma}$
with  $C:=C_G(y)$ is connected.     Note that $C$ is   $\sigma$ invariant.   
So $z = gy$ for some $g \in G$ and applying $\sigma$ we see that $z = \sigma(g) y$.
Thus,  $g^{-1} \sigma(g) \in C$.   Since $C$ is connected,  $g^{-1} \sigma(g)
= c^{-1} \sigma(c)$ for some $c \in C$.   Hence $gc^{-1} \in G_{\sigma}$ and
$gy = gc^{-1}y=z$, whence $y$ and $z$ are in the $G_{\sigma}$ orbit.
\end{proof}.

\section{Orbits}. \label{desc:orbits}

Let $k$ be a  field of characteristic $p$.   Let $m \le n$ be positive
integers.   We consider the action of $G:=GL_m(k) \times GL_n(k)$
acting on 
$W:=M_{m \times n}(k) \times M_{n \times m}$ given by 
$$
(X,Y)(A,B): =  (XAY^{-1}, YBX^{-1}).
$$

 As we noted earlier, two invariants of the orbit containing
$(A,B)$ are the conjugacy class of of $AB$ and the conjugacy class of $BA$.
 The indecomposable
orbits are described in Section \ref{quivers}.   It follows from the description
that the indecomposables are determined by these two invariants aside from
the cases where the orbits have representative $(I, J_d)$ and $(J_d,I)$ where
$J_d$ is a single nilpotent Jordan block of size $d$.  These orbits both have
invariants $J_d, J_d$.

We first recall (from the section on quivers):

\begin{lemma}  \label{connected}   Suppose that $k$ is algebraically closed.
The stabilizer of a point $(A,B)$ is connected.  
\end{lemma}

We now work over the finite field $F$ of $q^2$-elements.  Let $k$
be the algebraic closure of $F$.  

Let $\sigma$ denote the $q$-Frobenius map on $G$ (on each factor)
and let $\tau$ be the graph automorphism (i.e. the transpose inverse map)
on each factor of $G$.   Let $H = \langle G, \sigma, \tau \rangle$.  We extend
the action of $G$ to $H$ on $W$ as follows:

\begin{enumerate}
\item  $\sigma(A,B)= (\sigma(A), \sigma(B))$; and
\item $\tau(A,B)=(B^{\top}, A^{\top})$.
\end{enumerate}

It is straightforward to check this gives an action of $H$ on $W$.

Let $\rho = \sigma \tau$.   Then $\rho$ is an endomorphism of $G$
with fixed point group  $G_{\rho}=GU_m(q) \times GU_n(q)$.    
Note that $\rho$ fixes $(A,B)$ if and only if 
$A= \sigma(B^{\top})$ and $B=\sigma(A^{\top})$.  This implies that
$A$ and $B$ are fixed by $\sigma^2$ and so are matrices over $F$.
    Thus, we see that  $W_{\rho} = \{(A,A^*) | A \in M_{m \times n}(q^2)\}$.

By Corollary \ref{cor:orbits}(1), the orbits or $G_{\rho}$ on $W_{\rho}$ 
correspond to the $\rho$ stable orbits.
An invertible orbit containing $(I,B)$ is stable if and only  $B$ is similar to an invertible
Hermitian matrix (equivalently to a matrix over the field of $q$ elements).   The nilpotent
orbits are all defined over the prime field and so are all fixed by $\sigma$.  Thus, it suffices
to see how $\tau$ acts on the indecomposable nilpotent orbits.  The ones with dimension
vector $(d, d \pm 1)$ are uniquely determined by their dimension vector and so are invariant.
The other indecomposable nilpotent orbits have representatives $(I, J_d)$ or $(J_d,I)$
for some $d$ (where $J_d$ is the matrix consisting of single Jordan block) and clearly $\tau$ interchanges the orbits of $(I, J_d)$ and $(J_d, I)$.

Thus, we see that the orbits of $G_{\rho}$ on $W_{\rho}$ must have the same number
of terms for those orbits, no other condition on the nilpotent orbits and the invertible part
must correspond to a similarity class defined over the field of $q$ elements.  

We can now give a formula for the number of orbits.  As noted in the introduction, it 
suffices to do this in the case of nilpotent orbits. (i.e. $AA^*$ is nilpotent).    Let
$p(d)$ be the number of partitions of $d$.  

From the above discussion, we see that every nilpotent orbit on $M_{m \times n}(q^2)$ can be written uniquely as direct sum of $4$ pieces.    Choose nonnegative integers $m_1, m_2, m_3$
so that $m_1 + m_2 + 2m_3 = m$.   The first piece would correspond to the orbit (for 
the algebraic group) to   $\sum  R[\lambda_i]$ where the $\lambda_i$ form a partition of $m_1$,
the second piece to  $\sum  R'[\lambda_i]$ where the $\lambda_i$ form a partition of $m_2$
and the third piece to $\sum  J[\lambda_i] + J'[\lambda_i]$ where the $\lambda_i$ form a partition of $m_3$.   The fourth piece is just a sum of some number of copies of
$J'[1]$ to make sure the matrix has the right size.
If $n \ge 2m$,  all such decompositions are possible.  If $n < 2m$,  then there is an extra condition
(that the corresponding partitions for the columns is no more than $n$) and so there will be fewer orbits.   This gives:

\begin{cor} \label{nilpotent orbits}  Let $m, n$ be  positive integers.  Assume that 
$n \ge 2m$.   The number of
nilpotent orbits of $GU(m,q) \times GU(n,q)$ on $M_{m \times n}(F)$ is:
$$
\sum  p(m_1)p(m_2)p(m_3),
$$
where the sum is over all triples of non-negative integers $m_i$ with $m_1 + m_2
+2m_3 = m$.  
\end{cor}

One can easily work out a similar formula for the case that $n < 2m$. 
 
We can now easily prove  Theorem \ref{A*A versus AA*}
 about the possible canonical forms for $AA^*$ and $A^*A$.  
 The proof is by induction on the number of "indecomposable" summands.
 If the orbit is indecomposable the result is clear.  
 
 If $mn=0$, the result is clear.   So assume this is not the case. 
 
 Suppose first that $\alpha_1 \ne \beta_1$ and $\alpha_1 \beta_1 > 1$. 
 Suppose that $\alpha_1 > \beta_1$.    This forces the orbit to have
 an indecomposable summand $R[\alpha_1]$.     Remove this piece and use
 induction to conclude the result.    A symmetric argument handles the case
 that $\beta_1 > \alpha_1$.  
 
 Suppose that $\alpha_1=\beta_1 $.   If $\alpha_1=1$ and there is no condition to satisfy.
 So assume that $\alpha_1=d > 1$.   Then there must  be a summand of the
 form $J[\alpha_1] + J'[\alpha_1]$ or $R[\alpha_1] + R'[\alpha_1]$.   Remove this
 piece and use induction.
 
 Conversely, we can construct any admissible pair of partitions by reversing the
 argument.
 
 We now prove Corollary \ref{ATA versus AAT}.   So let $\alpha$
 and $\beta$ be partitions of $m$ and $n$.   Suppose that there 
 exists $A \in M_{m \times n}(F)$ with  $AA^{\top}$ nilpotent and
 $\mu(AA^{\top}) = \alpha$ and $\mu(A^{\top}A) = \beta$.
 We first work over a finite field 
 $F$.   Let $E$ be a quadratic extension of $F$.   Then $A^{\top}=A^*$
 (with respect to the nontrivial automorphism of $E/F$) and so 
 Theorem  \ref{A*A versus AA*} gives the result.   This shows that the only
 allowable  possibilities
 for $\alpha, \beta$ over the algebraic closure of $F$ are as given.  
 
 By a usual compactness argument (quite elementary in this case), this shows
 the same is true over any algebraically closed field and the result follows.   
 It is clear that not every possibility can be achieved over an arbitrary field
 (e.g.  over $\mathbb{R}$ or more generally a field in which $-1$ is not a square).
 If $-1$ is a square, in fact every possibility can be realized as we shall see later.

\section{Further Remarks}

We first give an elementary proof of the decomposition into invertible and nilpotent parts.
The result already follows from the results of the previous section on orbits. 

Let $F$ be a finite field with $q^2$ elements.   Let $V, W$ be finite dimensional vector
spaces equipped with nondegenerate Hermitian forms.

\begin{lemma}   Let $T:V \rightarrow W$ be an $F$-linear transformation.
Then there exist canonical decompositions $V = V_1 \perp V_2$ and
$W=W_1 \perp W_2$ with $TV_1=W_1$  bijective and $TV_2 \subset W_2$  
and $T^*T$ nilpotent on $V_2$.
\end{lemma}

\begin{proof} Let $S=T^*T$.  Then $S$ is an Hermitian operator on $V$.
Thus, $V=V_1 \perp V_2$ where $S$ is invertible on $V_1$ and $S$ is 
nilpotent on $V_2$ (e.g. $V_1$ is the image of sufficiently large power  $M$ of
$S$ and $V_2$ is the kernel of the same power).

It follows that $\ker(T) \cap V_1=0$.   Let $W_1=TV_1$.  We claim that
$W_1$ is nonsingular.   If $(Tv_1, Tv_2)=0$ with $v_1, v_2 \in V_1$, 
then $(T^*Tv_1, v_2)=0$
and if this holds for all $v_1 \in V_1$, then as $S=T^*T$ is bijective on $V_1$
we see that $v_2 \in V_1 \cap V_2^{\perp}=0$ and the claim is proved.

A similar argument shows that $W_1$ is orthogonal to $TV_2$ and so 
$TV_2 \subset W_2:= W_1^{\perp}$.   This completes the proof.
\end{proof} 

Suppose we are given $T$.   Let $X=T^*T$ and $Y=TT^*$.  Note
that the orbit of $T$ contains $S$ with $S^*S$ any Hermitian matrix
similar to $X$ and $SS^*$ similar to $Y$.    In particular,  if $X$ and $Y$
are similar (whence we are working with square matrices), 
there is an element in the orbit with $S^*S = SS^*$ (and so $S$ commutes
with $S^*$, i.e. $S$ is normal).   We record this:

\begin{lemma} Let $T:V \rightarrow V$ be an $F$-linear transformation.
Then the orbit of $T$ contains a normal matrix if and only if 
$T^*T$ and $TT^*$ are similar.   In particular, this is always the 
case if $T$ is invertible.
\end{lemma}

It follows from the classification in the last section that it is possible for
$T^*T$ to be a single Jordan block of size $n>1$ and then necessarily 
$TT^*$ would have two Jordan blocks of sizes $n-1$ and $1$ and so there
are no normal matrices in the orbit.

If $T$ is invertible, by the classification and the previous lemma, we can reduce to the case
where $H:=T^*T$ is non-derogatory (i.e. the minimal and characteristic polynomials coincide)
 and $T$ is normal.   Thus,  $T=f(H)$ is a polynomial
in $H$ and $T^*={\bar f}(H)$.   

For the nilpotent case, there do not seem to be "nice" canonical normal forms, but
we give a different proof for the existence in the indecomposable case.  

\begin{lemma}  Fix a positive integer $d$.   If $H$ is a Hermitian $d \times d$ matrix
in $M_{d}(F)$ of rank $d-1$, then there exists a column vector $v$ of size $d$ such that
$H -  vv^*$ is invertible.  
\end{lemma}

\begin{proof}   Suppose $Hw = vv^*w$  with $w \ne 0$.   Then either 
$w \in \ker(H)$ and $v^*w=0$ or $v$ is in the image of $H$.  Thus, we just
need to guarantee that there is a vector $v$ outside the hyperplane $\mathrm{im}(H)$
and $v^*w \ne 0$ for $0 \ne w \in \ker(H)$.  This is true since a finite dimensional
vector space is not the union of two proper subspaces.
\end{proof}

\begin{corollary}  Fix a positive integer $d$.   Then there exists $B \in M_{d \times d+1}(F)$
such that $BB^*$ and $B^*B$ are nilpotent matrices with a single Jordan block.
\end{corollary}

\begin{proof}    By Lemma \ref{hermitian1},  there exists an Hermitian $d \times d$ matrix
over $F$ so that $H$ is nilpotent with a single Jordan block.    
By the previous result,  there exist a column vector
of size $d$ with $H + vv^*$ invertible.   So $H - vv^*= AA^*$ for some $ A \in GL_d(F)$
by Lemma \ref{hermitian2}.

Set $B = [ A | v ] \in M_{d \times d+1}(F)$.   Then. $BB^* = AA^*+ vv^* =H$ is nilpotent
with a single Jordan block.  Thus  $B^*B$ is nilpotent.   Since $B$ has rank $d$, 
it follows that $B^*B$ has a $1$ dimensional kernel (since $B^*$ is injective and $B$
has a $1$ dimensional null space).   Thus,  $B^*B$ is a single Jordan block.
\end{proof}

Now assume that we are over an arbitrarily field $K$ and $*$ is replaced by transpose.
The same proof applies with $H$ replaced by a symmetric matrix.  We note any that nonzero square
matrix (in particular any nilpotent matrix) is similar to a symmetric matrix (over some extension
field) \cite{TZ}.   Over some extension field any symmetrix matrix (in characteristic $2$,
we also need to ensure the matrix is not skew) can be written at $AA^{\top}$ (again
possibly passing to an extension field).   This yields:

\begin{corollary}  Let $K$ be an algebraically closed field.
Fix a positive integer $d$.   Then there exists $B \in M_{d \times d+1}(K)$
such that $BB^{\top}$ and $B^{\top}B$ are both nilpotent matrices with a single
Jordan block.
\end{corollary}

This shows that the necessary condition of Corollary \ref{ATA versus AAT} is also
sufficient at least over some extension field.  In fact, all that is required is that $-1$
is a square in the field (and thinking of the $1 \times 2$ this is necessary).   We sketch
the proof of this.

Let $J$ be a nilpotent $d \times d$ matrix over a field $K$ that is a single
Jordan block (taking $J$ to have $1$'s just above the diagonal).  Let. $S$
be the symmetric matrix with the only nonzero entries being $1$ on the anti-diagonal.
Then $SJ=J^{\top}S$ and so $J$ is a self adjoint operator with respect to the
bilinear form given by $S$. 

If the characteristic is not $2$ and $-1$ is a square, then 
 then $S = X^{\top}X$ for some
$X$. (i.e $S$ is congruent to $I$).   Thus,  some conjugate $L$ of $J$ is symmetric
and moreover the Gram matrix of this form is easily seen to be congruent to
$\mathrm{diiag}(1, \ldots, 1, 0)$ (this is again using that $-1$ is a square).

In characteristic $2$, the same argument applies for $d$ odd, but if $d$ is even,
then $S$ is skew symmetric and the argument does not apply.  In that case,
we replace $S$ by $S(I+J)$ which will be symmetric and invertible with nonzero
trace.  In a perfect field of characteristic $2$, any invertible symmetric matrix
that is not skew can be written as $X^{\top}X$.  

Thus. $L + vv^{\top}$ is congruent to $I$
and so can be written as $Y^{\top}Y$.  Now argue as above.

\end{document}